\documentclass[11pt]{amsart}

\usepackage{times,amssymb,amsmath,hyperref}
\usepackage{amscd,amsfonts,amsmath,amssymb,epsf,rlepsf,amsthm,amstext,latexsym}
\usepackage{enumerate,pifont,graphicx}

\usepackage{color}
\usepackage[english]{babel}
\usepackage[T1]{fontenc}
\usepackage{epsfig}
\usepackage[latin1]{inputenc}

\usepackage{enumerate,latexsym}
\usepackage{latexsym}
\usepackage{amsmath,amssymb}
\usepackage{graphicx}

\newcommand{\Int}{\mbox{\rm Int}}

\newtheorem{theorem}{Theorem}[section]
\newtheorem{lemma}[theorem]{Lemma}

\newtheorem{claim}[theorem]{Claim}
\newtheorem{proposition}[theorem]{Proposition}
\newtheorem{remark}[theorem]{Remark}
\theoremstyle{definition}
\newtheorem{definition}{Definition}[section]

\theoremstyle{remark}

\def\square{\hfill${\vcenter{\vbox{\hrule height.4pt \hbox{\vrule
width.4pt height7pt \kern7pt \vrule width.4pt} \hrule height.4pt}}}$}

\newcommand{\PI}{\partial_{\infty}}
\newcommand{\Si}{\partial_\infty \mathbb{H}^3}

\newcommand{\ov}{\overline}
\newcommand{\cF}{\mathcal{F}}

\newcommand{\BZ}{\mathbb{Z}}

\newcommand{\ve}{\varepsilon}
\newcommand{\T}{\mathcal{T}}
\renewcommand{\l}{\lambda}

\newcommand{\cL}{\mathcal{L}}

\newcommand{\wh}{\widehat}
\newcommand{\wt}{\widetilde}
\newcommand{\C}{\mathcal{C}}
\newcommand{\cC}{\mathcal{C}}
\newcommand{\Z}{\mathcal{Z}}
\newcommand{\ZZ}{\mathbb{Z}}
\newcommand{\N}{\mathbb{N}}

\newcommand{\R}{\mathbb{R}}
\newcommand{\rth}{\R^3}
\newcommand{\HH}{\mathbb{H}}
\newcommand{\ben}{\begin{enumerate}}
\newcommand{\bit}{\begin{itemize}}
\newcommand{\een}{\end{enumerate}}
\newcommand{\eit}{\end{itemize}}
\newcommand{\G}{\Gamma}
\newcommand{\ed}{\end{document}}

\begin{document}
\title{Non-properly Embedded $H$-Planes in $\mathbb{H}^3$}
\author{Baris Coskunuzer} \thanks{The first author is partially supported by TUBITAK 2219 Grant, and Fulbright Grant.}
\author{William H. Meeks III} \thanks{The second author was supported in part by NSF Grant DMS -
   1309236. Any opinions, findings, and conclusions or recommendations
   expressed in this publication are those of the authors and do not
   necessarily reflect the views of the NSF}
\author{Giuseppe Tinaglia} \thanks{The third author was partially supported by EPSRC grant no. EP/I01294X/1}
\address{Department of Mathematics \\ MIT \\ Cambridge, MA 02139}
\address{Department of Mathematics \\ Koc University \\ Sariyer, Istanbul 34450 Turkey}
\email{bcoskunuzer@ku.edu.tr}
\address{Department of Mathematics \\ University of Massachusetts \\ Amherst, MA 01002}
\email{profmeeks@gmail.com}
\address{Department of Mathematics \\ King's College London, London}
\email{giuseppe.tinaglia@kcl.ac.uk}

\begin{abstract}
For any $H \in [0,1)$, we construct  complete, non-proper, stable,
simply-connected surfaces with constant mean curvature $H$ embedded in hyperbolic three-space.
\end{abstract}


\maketitle

\section{Introduction}

In
their ground breaking work
~\cite{cm35} Colding and
Minicozzi proved that  complete minimal surfaces  embedded
in $ \rth$ with finite topology are proper.
 Based on the techniques in \cite{cm35},  Meeks
and Rosenberg~\cite{mr13} then proved that
complete minimal surfaces embedded in $\mathbb{R}^3$ are proper, if they
have positive injectivity radius; since complete, immersed finite topology minimal
surfaces in $\rth$ have positive injectivity radius,
their result generalized Colding and Minicozzi's work.

Recently  Meeks and Tinaglia~\cite{mt7} proved that  complete
 constant mean curvature surfaces  embedded in
$\rth$  are proper if they
have finite topology or have positive injectivity radius.
In fact their results in $\rth$ should generalize
to show that a complete embedded
surface $\Sigma$ of constant mean curvature $H\in [1,\infty)$
in a complete hyperbolic
three-manifold is proper if $\Sigma$ has finite topology
or it is connected and has positive injectivity
radius; this is work in progress in~\cite{mt11}.

In contrast to the above  results, in this paper we prove the
following existence theorem for non-proper, complete
simply-connected surfaces embedded in $\mathbb{H}^3$
with constant mean curvature $H\in [0,1)$.
See the Appendix where a description
of  the spherical catenoids appearing in the next theorem is  given.

\begin{theorem}\label{main} For any $H\in [0,1)$ there
exists a complete simply-connected surface $\Sigma_H$ embedded
in $\mathbb{H}^3$ with constant mean curvature $H$
satisfying the following properties:
\begin{enumerate}
\item The closure of $\Sigma_H$ is a lamination with three
leaves, $\Sigma_H$, $C_1$ and $C_2$,
where $C_1$ and  $C_2$
are stable spherical catenoids of constant mean curvature $H$ in
$\mathbb{H}^3$ with the same axis $L$ of revolution.
In particular, $\Sigma_{H}$ is not properly embedded in $\HH^3$.
\item The asymptotic boundary of $\Sigma_H$ is a pair of
embedded curves  in $\partial_{\infty} \HH^3$ which spiral
into the union  of the round circles which are the
asymptotic boundaries of $C_1$ and $C_2$.
\item Let $K_L$ denote the  Killing field generated by
rotations around $L$. Every
integral curve of $K_L$ that lies in the region between
$C_1$ and $C_2$ intersects $\Sigma_H$
transversely in a single point.  In particular,
the closed region between
$C_1$ and $C_2$ is foliated by  surfaces   of constant mean curvature $H$,
where the leaves are
$C_1$ and $C_2$ and  the rotated images $\Sigma_H ({\theta})$ of
$\Sigma$ around $L$ by angle $\theta\in [0,2\pi)$.
\end{enumerate}
\end{theorem}

Previously Coskunuzer~\cite{cosk1} constructed an
example of a non-proper, stable, complete minimal plane
in $\HH^3$ that can be roughly described as a
collection of "parallel" geodesic planes connected via
"bridges at infinity". However, his
techniques do not generalize to construct non-proper,
non-zero constant mean curvature planes
in $\mathbb{H}^3$.

There is a general conjecture related to Theorem~\ref{main}
and the previously stated positive properness results. This conjecture states
that if $X$ is a simply-connected,
homogeneous three-manifold with Cheeger constant Ch$(X)$, then for
any $H\geq\frac12$Ch$(X)$, every  complete, connected
$H$-surface embedded in $X$ with positive injectivity
radius or  finite topology  is proper.
In the case of the Riemannian product $X=\HH \times \R$,
then Ch$(X)=\frac12$, and the validity of this
conjecture would imply that every complete, connected
$H$-surface embedded in $X$ with positive injectivity
radius or  finite topology  is properly embedded when $H\geq \frac12$.
In view of this conjecture and Theorem~\ref{main},  it is natural to ask the question:

\begin{quote}
{\em Given $H\in [0,\frac12)$, does there exist a complete, non-properly embedded
$H$-surface of finite topology in
$\HH \times \R$ ? }
\end{quote}

When $H=0$,
Rodr\'{\i}guez and Tinaglia~\cite{rodt1} have  constructed non-proper, complete
minimal planes embedded in $\mathbb{H}\times \R$.
However, their construction does  not generalize to
produce complete, non-proper planes embedded in $\HH \times \R$
with non-zero constant mean curvature.

\nocite{bt1}

\section{An outline of the construction}

In this section, we outline the construction of the
examples described in Theorem~\ref{main}.
Throughout the paper, we refer to an oriented surface
embedded in $\mathbb{H}^3$ with constant
mean curvature $H$   as an {\em $H$-surface}, and call it an {\em $H$-disk} if it is
simply-connected. After possibly reversing the orientation of an $H$-surface, we will always
assume  $H\geq 0$. Given a domain $\Omega\subset \mathbb{H}^3$  with
smooth boundary $\partial \Omega$, we say that
 $\partial \Omega$ is $H_0$-convex, $H_0\geq 0$, if after orienting
$\partial \Omega$ so that its unit normal is pointing into $\Omega$, then
$\inf_{\partial \Omega}H_{\partial \Omega}\geq H_0$, where $H_{\partial \Omega}$
denotes the mean curvature function of $\partial \Omega$.

We will work in $\mathbb{H}^3$ using the Poincaré ball model, that is
we consider $\mathbb{H}^3$ as the unit ball in $\rth$ and   its
ideal boundary $\partial_\infty \HH^3$ at infinity corresponds to the boundary of the ball.
Fix $H$ in $[0,1)$.
Given  $\lambda_1>0$ sufficiently large and $\lambda_2>\lambda_1$,
for any $\l \in [\l_1,\l_2]$, there exists a unique spherical $H$-catenoid
$\C^{\lambda}$ whose distance to its rotation axis,
which we assume is the   $z$-axis, is $\lambda$ in the hyperbolic metric, is invariant under
reflection in the $(x,y)$-plane  and the mean curvature vectors of
$\C^{\lambda}$ point toward the $z$-axis; see Figure~\ref{univcover1}
and the discussion in the Appendix for further details.

Let $W \subset \HH^3$ denote the closed region between
$\C^{\l_1}$ and $\C^{\l_2}$. By
Proposition~\ref{foliation} in the Appendix, the collection
$\cF=\{\C^{\lambda} \mid  \l \in [\l_1,\l_2] \}$ is a
foliation of $W$.
We will identify $W$ topologically with
$[\l_1,\l_2 ]\times \mathbb{S}^1 \times \R$ and its boundary
consists of the two $H$-surfaces
$\C^{\lambda_1}$ and $\C^{\lambda_2}$; in the next section this
identification is made explicit.

For $\lambda_1$ sufficiently large and for a fixed $\lambda_2>\lambda_1$
that is sufficiently close to $\lambda_1$,
we will construct the surface $\Sigma_H\subset W$ described in Theorem~\ref{main}
by creating a sequence of compact $H$-disks in $W$ whose interiors
converge to $\Sigma_H$ on compact subsets of $\Int(W)$. To do this, we will consider the
universal cover $\wt{W}=[\l_1,\l_2 ]\times\R \times \R$ of $W$,
which is an infinite slab with boundary
$H$-planes $\wt{\C}^{\lambda_1}$ and $\wt{\C}^{\lambda_2}$.
By construction, the mean curvature vectors  of the surface
 $\wt{\C}^{\lambda_2} \subset \partial \wt{W}$ point into
$\wt{W}$, while the mean curvature vectors of $\wt{\C}^{\lambda_1}$ point out of
$\wt{W}$; see Figure~\ref{univcover1}.

To create the compact sequence of $H$-disks we
first exhaust $\wt{W}$ by a certain increasing sequence of
compact domains $\Omega_n \subset \Omega_{n+1}$ such that
$\partial \Omega_n\setminus (\partial \Omega_n\cap \wt{\C}^{\lambda_1})$ is $H$-convex.
Next we choose an appropriate sequence of simple closed curves
$\Gamma_n$ on $\partial \Omega_n$ so that each $\Gamma_n$ is the boundary of an
$H$-disk $\Sigma_n$ embedded in $\Omega_n$, with each such
disk being a graph over its natural projection to
$[\l_1,\l_2 ]\times\{0\}\times  \R$; see Figure~\ref{Gamma_n}. A  compactness argument
then gives that a subsequence of the projected interiors of the surfaces $\Pi(\Sigma_n)\subset W$
converges to a complete
$H$-disk ${\Sigma}_H$ embedded in $ \Int(W)$, which we prove is a entire
graph over $(\l_1,\l_2 )\times\{0\}\times  \R$.
Finally we will show that $\Sigma_H$    satisfies the other conclusions
of Theorem~\ref{main}.

\section{The examples}

\subsection{The construction of the compact exhaustion of $\wt{W}$}
We begin by explaining in detail the construction of
the domains $\Omega_n$ briefly described in the previous
section. For the remainder of the paper we fix a particular  $H\in[0,1)$.

Let $c_H>0$ be the constant described in Lemma~\ref{Hcatenoid}
and Proposition~\ref{foliation} in the Appendix.
As described in the previous section, given $\lambda_2>\lambda_1\geq c_H$,
$W$  denotes the
closed region between $\C^{\l_1}$ and $\C^{\l_2}$. The
number $\lambda_2$ will be fixed later. The region $W$
is foliated by the collection
$\cF=\{\C^{\lambda} \mid  \l \in [\l_1,\l_2] \}$ of
spherical $H$-catenoids.
By using the
foliation $\cF$ of $W$, we
introduce {\em cylindrical coordinates} $(\lambda,\theta,z)$
on $W$ as follows.  The coordinate
$\lambda$ indicates that the point is in
$\C^\lambda$. The coordinate $\theta\in [0,2\pi)$
parameterizes the core circles of the catenoids in $\cF$ and corresponds to $\theta$
in cylindrical coordinates in $\rth$.
Finally, the $z$-coordinate of a point  $(\lambda,\theta,z)$ represents
the signed intrinsic distance on $\C^\lambda$ to the core circle of the catenoid
$\C^\lambda$ which lies in the $(x,y)$-plane, where the sign is taken to be positive if
the $z$-coordinate of the point in the ball model is positive, and otherwise
it is negative. Note that this choice of $z$-coordinate
is different from the one
of the ball model. It is clear that for points of $W$, $\lambda\in [\lambda_1,
\lambda_2]$, $\theta\in [0,2\pi)$ and $z\in (-\infty,\infty)$
in these coordinates; see Figure~\ref{univcover1}.

Let $\wt{W}$ be the universal cover of $W$, which is
topologically an infinite slab $[\l_1,\l_2] \times \R \times  \R$ with boundary
$H$-planes $\wt{\C}^{\lambda_1}$ and $\wt{\C}^{\lambda_2}$. We will use
the {\em induced} coordinates $(\lambda, \wt{\theta} , z)$
in $\wt{W}$. Namely, if $\Pi:\wt{W}\to W$ is
the covering map, then
\[
\Pi(\lambda,\wt{\theta}, z)=(\lambda,\wt{\theta} \;\; {\rm mod} \; 2\pi, \, z),
\]
i.e.,  $\Pi$ keeps fixed the $\lambda$ and $z$ coordinates,
and sends  $\wt{\theta}\in (-\infty,\infty)$ to the point
($\wt{\theta} \ \; {\rm mod} \; 2\pi$) corresponding
to a point in the core circle of the catenoid;
see Figure~\ref{univcover1}. $\wt W$ is endowed with the
metric induced by $W$ and  in these coordinates, for any
$\theta_0\in (-\infty,\infty)$, the map
\[
T_{\theta_0}\colon \wt{W}\to\wt{W},\quad T_{\theta_0}(\lambda, \wt{\theta}, z)=(\lambda, \wt{\theta}+\theta_0, z)
\]
is an isometry of $\wt{W}$ as it is induced by the isometry of
$\mathbb{H}^3$ which is a rotation by angle $\theta_0$
about the $z$-axis. In particular, $T_{2\pi n}$ is a
covering transformation for any $n$. We let $\partial_\theta$ denote the Killing field in
$W$ generated by the rotations about the $z$-axis and denote by $\partial_{\wt \theta}$
the related Killing field in $\wt W$ generated by the one-parameter group of
isometries $\{T_\theta\}_{\theta\in \R}$.

\begin{figure}[t]
\begin{center}
$\begin{array}{c@{\hspace{.2in}}c}

\relabelbox  {\epsfxsize=2in \epsfbox{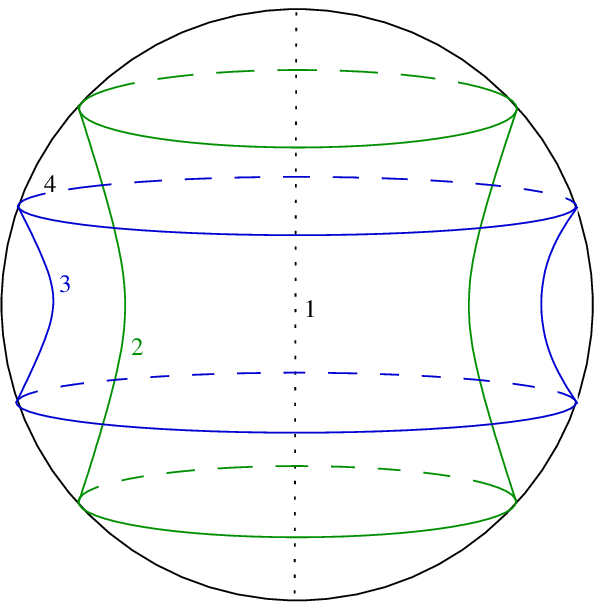}}
\relabel{1}{\footnotesize $z$-axis} \relabel{2}{\footnotesize \color{green} $\C^{\lambda_1}$}
\relabel{3}{\footnotesize \color{blue} $\C^{\lambda_2}$} \relabel{4}{\footnotesize $W$}
\endrelabelbox &

\relabelbox  {\epsfxsize=2.5in \epsfbox{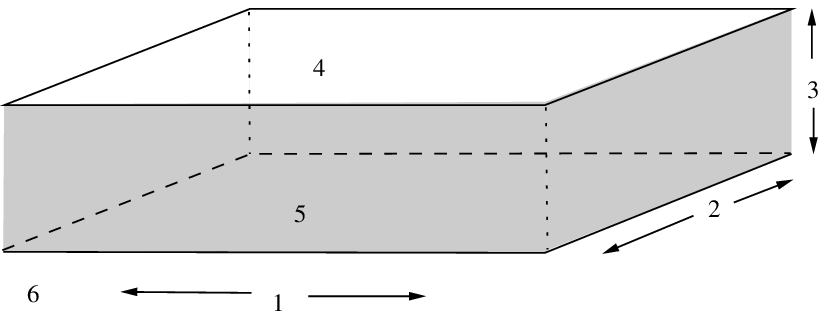}}
\relabel{1}{\footnotesize $\wt{\theta}$} \relabel{2}{\small $z$} \relabel{3}{\footnotesize
$\lambda$} \relabel{4}{\footnotesize $\wt{\C}^{\lambda_2}$} \relabel{5}{$\wt{W}$} \relabel{6}{\footnotesize $\wt{\C}^{\lambda_1}$}
\endrelabelbox \\ [0.4cm]
\end{array}$

\end{center}

\caption{ The induced coordinates $(\lambda, \wt{\theta}, z)$ in $\wt{W}$.} \label{univcover1}

\end{figure}

Since when $H>0$ the mean curvature vectors
of the boundary of $W$ point towards the rotation
axis, the mean curvature vectors point
into $\wt{W}$ on $\wt{\C}^{\lambda_2}$ and  out of $\wt{W}$ on
$\wt{\C}^{\lambda_1}$; thus, when considered to be a part of $\partial \wt W$,
$\wt \C^{\lambda_2}$ is $H$-convex, while $\wt \C^{\lambda_1}$ is not.

\begin{figure}[t]
\begin{center}
$\begin{array}{c@{\hspace{.2in}}c}

\relabelbox  {\epsfxsize=2in \epsfbox{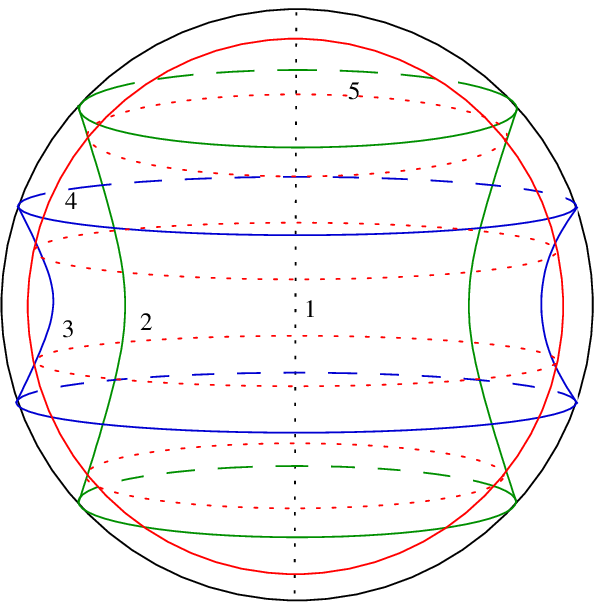}}
\relabel{1}{\footnotesize $z$-axis}
\relabel{2}{\footnotesize \color{green} $\C^{\lambda_1}$}
\relabel{3}{\footnotesize \color{blue} $\C^{\lambda_2}$}
\relabel{4}{\footnotesize $W_n$}
\relabel{5}{\footnotesize \color{red} $B_{R_n}$} \endrelabelbox

&

\relabelbox  {\epsfxsize=2.5in \epsfbox{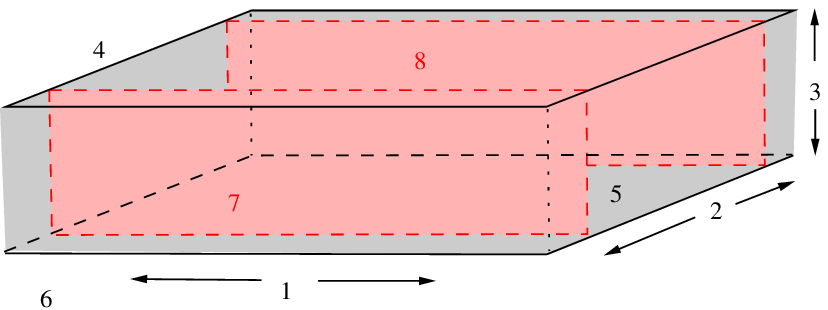}}
\relabel{1}{\footnotesize $\wt{\theta}$}
\relabel{2}{$z$}
\relabel{3}{\footnotesize $\lambda$}
\relabel{4}{\footnotesize $\wt{\C}^{\lambda_2}$}
\relabel{5}{\footnotesize $\wt{W}_n$}
\relabel{6}{\footnotesize $\wt{\C}^{\lambda_1}$}
\relabel{7}{\footnotesize \color{red} $\wt{\Z}_n^+$}
\relabel{8}{\footnotesize \color{red} $\wt{\Z}_n^-$} \endrelabelbox \\ [0.4cm]

\end{array}$

\end{center}

\caption{\label{univcover2} $W_n = W\cap B_{R_n}$ and $\wt{W}_n$ denotes its
universal cover. Note that $\partial \wt{W}_n\subset
\wt{\C}^{\lambda_1}\cup\wt{\C}^{\lambda_2}\cup\Z^+_n\cup\Z^-_n$.}

\end{figure}

As explained in Section 2, our next goal is to exhaust
$\wt{W}$ by an increasing sequence of
compact domains $\Omega_n \subset \Omega_{n+1}$
such that $\partial \Omega_n\setminus
(\partial \Omega_n\cap \C^{\lambda_1})$ is $H$-convex.
Let $R_n\nearrow \infty$ as $n\nearrow\infty$
and let $B_{R_n}$ be the closed geodesic ball in
$\mathbb{H}^3$ with center the origin. Let
$W_n = W\cap B_{R_n}\subset  W $ and let $\wt{W}_n$ be the universal
cover of $W_n$; see Figures~\ref{univcover1} and \ref{univcover2}.
Assume $R_1$ is chosen sufficiently large so that
every  $\wt{W}_n$ can be viewed as an
infinite tube in $\wt W$ which is bounded in the $z$-direction,
but unbounded in $\wt\theta$-direction. Then there
exists a sequence of  bounded continuous positive functions
\[
Z_n\colon [\lambda_1,\lambda_2]\to (0,\infty), \quad Z_{n+1}>Z_n
\] such that
\[
\wt W_n =\{(\lambda, \wt{\theta}, z)\in\wt W \text{ with } z\in [-Z_n(\lambda),Z_n(\lambda)]\}.
\]
Note that $Z_n$ does  not depend on $\wt\theta $ because $B_{R_n}$ is rotationally
symmetric. Let $\Z_n^\pm$ be the two annular
components of $\partial B_{R_n} \cap W$. The preimages of the
surfaces $\Z_n^\pm$ in the universal cover $\wt{W}$ are
\[
\wt{\Z}_n^\pm:=\{(\lambda, \wt{\theta}, \pm Z_n(\lambda)) \mid \lambda\in [\lambda_1,\lambda_2],
\wt \theta \in (-\infty,\infty)\}.
\]
Since the mean
curvature of $\partial B_{R_n}$ with inner pointing
unit normal is strictly greater than one, $\wt{\Z}_n^+\cup\wt{\Z}_n^-$
is $H$-convex as part of the boundary of $\wt W_n$.

The final and most difficult step in defining the piecewise-smooth
compact domains $\Omega_n$ is to bound
$\wt{W}_n$ in the $\wt \theta$-direction. In order
to do this, we will again use spherical $H$-catenoids.
Let $P^+=(0,-\ve,p_3)$ be a point on
$\Si$ for some small $\ve>0$ and let $P^-=(0,-\ve,-p_3)\in\Si$ be its
symmetric point in $\Si$ with respect to the $(x,y)$-plane.
Let $\gamma$ be
the geodesic in $\HH^3$ connecting $P^+$ and $P^-$ and let $\phi$
be the hyperbolic translation  along  the $y$-axis in
the ball model for $\HH^3$ and that
maps the $z$-axis to $\gamma$.

For $\ve>0$ chosen sufficiently small,
the asymptotic boundary circles of $\wh{\C}^{\lambda_1}=\phi(\C^{\l_1})$ intersect transversely the
 asymptotic boundary circles of ${\C}^{\lambda_1}$, and
  $\widehat{\C}^{\lambda_1}$
 intersects $\C^{\lambda_1}$ transversely with
$\wh{\C}^{\lambda_1}\cap \C^{\lambda_1}=l^+\cup l^-$, where $l^\pm$ is a pair of
infinite ``vertical" arcs in the intersecting
catenoids. 
See the proof of Proposition~\ref{prop:5.3}  in the Appendix for the details on the existence  of $l^\pm$
and for some details on the next argument.
Then after choosing $\lambda_2-\lambda_1$
sufficiently small, the intersection
$\wh{C}^{\lambda_1}\cap W$ consists of  two  thin infinite
strips $\T_+ $ and $\T_-$, where $\T_\pm$ looks like  $ l^\pm \times
(\lambda_1,\lambda_2)$. The strips $\T_+ $ and $\T_-$
separate $W$ into two components, say $W^+$ and
$W^-$, and $\T_+ \cup \T_-$ is $H$-convex as
boundary of one of these two components, say
$W^+$.

%
%
%
%

Notice that the strips $\T_+$ and $\T_-$ have
infinitely many lifts $\{\wt{\T}^n_+\}_{n\in \ZZ}$
and $\{\wt{\T}^n_-\}_{n\in \ZZ}$ in $\wt{W}$. In particular, if we fix a lift
$\wt{\T}^0_+$, then $\wt{\T}^n_+=T_{2\pi n}(\wt{\T}^0_+)$
for any $n\in \BZ$.  Similarly,
the same is true for $\wt{\T}^n_-$. We fix the lifts
$\wt{T}^0_+$ and $\wt{\T}^0_-$ in $\wt{W}$ so that
the mean curvature vectors are pointing
 into the region that they bound, and so that there is no other lift
$\wt{\T}^n_\pm$ between them. Then there exists a
function $ G\colon [\lambda_1,\lambda_2]\times
(-\infty,\infty) \to (0,\pi)$ such that
\[\wt{\T}^0_+=\{(\lambda, G(\lambda, z), z)\} \text{ and } \wt{\T}^0_-=\{(\lambda, -G(\lambda, z), z)\}\]
Moreover, let $G_n(\lambda,z)=G(\lambda,z) +2\pi n$. Then,
$\wt{\T}^n_+=\{(\lambda, G_n(\lambda, z), z)\}$
and $\wt{\T}^n_-=\{(\lambda,
-G_n(\lambda, z), z)\}$.

Finally, let $\Omega_n$ be the region in $\wt{W}_n$
between $\wt{\T}^n_+$ and $\wt{\T}^{-n}_-$. In particular,
\[
\Omega_n=\{ (\lambda,\wt{\theta}, z) \in\wt{W} \ | \ \wt{\theta}\in [-G_n(\lambda,z), G_n(\lambda,z)]
 \mbox{ and } z\in [-Z_n(\lambda),Z_n(\lambda)]
\}
\]
Hence, we have obtained an exhaustion of $\wt{W}$ by compact regions with the property
that $\partial \Omega_n\setminus (\partial \Omega_n\cap
\wt\C^{\lambda_1})$ is $H$-convex; see Figure~\ref{Omega_n}.

\begin{figure}[b]
\begin{center}
$\begin{array}{c@{\hspace{.1in}}c}

\relabelbox  {\epsfxsize=2.5in \epsfbox{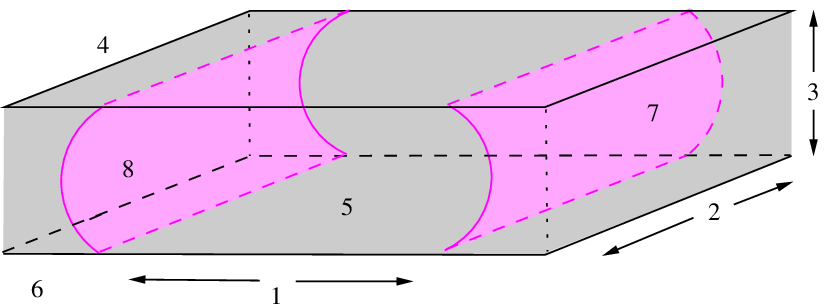}} \relabel{1}{\footnotesize $\wt{\theta}$} \relabel{2}{$z$} \relabel{3}{\scriptsize $\lambda$}
\relabel{4}{\footnotesize $\wt{\C}^{\lambda_2}$} \relabel{5}{\footnotesize $\wt{W}$} \relabel{6}{\footnotesize $\wt{\C}^{\lambda_1}$}
\relabel{7}{\footnotesize $\wt{\T}^n_+$} \relabel{8}{\footnotesize $\wt{\T}^{-n}_-$} \endrelabelbox &

\relabelbox  {\epsfxsize=2.5in \epsfbox{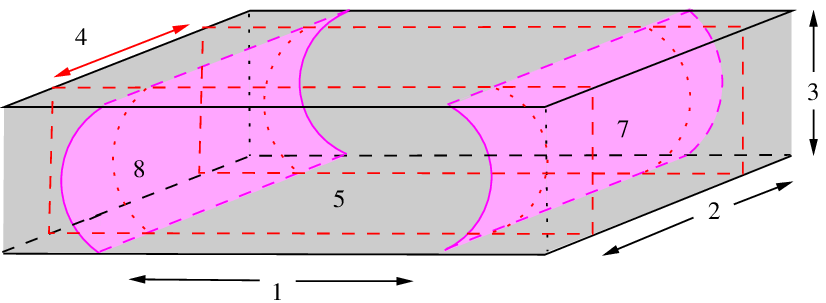}} \relabel{1}{\footnotesize $\wt{\theta}$} \relabel{2}{$z$} \relabel{3}{\scriptsize $\lambda$}
\relabel{4}{\scriptsize $\wt{W}_n$} \relabel{5}{\footnotesize $\Omega_n$}
\relabel{7}{\footnotesize $\wt{\T}^n_+$} \relabel{8}{\footnotesize $\wt{\T}^{-n}_-$} \endrelabelbox \\
[0.4cm]
\end{array}$

\end{center}

\caption{ $\Omega_n$ is the region in $W_n$ between $\wt{\T}^n_+$ and $\wt{\T}^{-n}_-$.} \label{Omega_n}

\end{figure}

\subsection{The sequence $\Sigma_n$ of graphical $H$-disks}

Our aim in this section is to construct a sequence of
compact $H$-disks $\Sigma_n \subset \Omega_n$ with
$\partial \Sigma_n \subset \partial \Omega_n$, which are
$\wt{\theta}$-graphs over their projections to
$[\l_1,\l_2 ]\times\{0\}\times  \R$.


Let
\[
\partial_* \Omega_n:=\partial \Omega_n\setminus
(\partial \Omega_n\cap \{\wt\C^{\lambda_1}\cup \wt\C^{\lambda_2}\})
\]
and let $\gamma$ be a piecewise smooth, embedded,
simple closed curve in $\partial_* \Omega_n$ that does
not bound a disk in $\partial_* \Omega_n$. Recall
that $\partial_* \Omega_n$ is piecewise smooth and
$H$-convex as part of the boundary of $\Omega_n$,
since the dihedral angles are less than $\pi$ at the corners.

Consider the following variational problem. Let $M$
be a compact surface embedded in $\Omega_n$
with $\partial M=\gamma\subset \partial_*
\Omega_n$. Since  $\Omega_n$ is simply-connected, $M$
separates $\Omega_n$ into two regions,
i.e., $\Omega_n - M = \Omega_M^+\cup\Omega_M^-$ where
$\Omega^+_M$ denotes the region that contains
$\wt\C^{\lambda_2}\cap\Omega_n$. Let $A(M)$ denote the area
of $M$ and let $V(M)$ denote the volume of the
region $\Omega_M^+$.  Then, let
\begin{equation}\label{Func-currents}
I(M)=A(M)+2HV(M).
\end{equation}

By working with integral currents, it is known that for any simple closed
essential curve $\gamma_n$ in $\partial_* \Omega_n$
there  exists a smooth (except at the 4 corners of $\gamma_n$), compact,
embedded $H$-surface $\Sigma_n\subset W_n$ with  $\text{Int}(\Sigma_n)\subset \text{Int}(W_n)$ and
$\partial\Sigma_n=\gamma_n$.
In fact,  $\Sigma_n$ can be chosen to be,
and we will assume it is, a minimizer for this variational
problem, i.e.,
$I(\Sigma_n)\leq I(M)$ for any
$M\subset \Omega_n$ with $\partial M =\gamma_n$; see for instance~\cite{alr1,ton1}.
Moreover,  $\Sigma_n$ separates $\Omega_n$ into two
regions and 
the mean curvature vectors of $\Sigma_n$ points ``down,'' namely into $\Omega_{\Sigma_n}^-$.

If $\lambda_{\Sigma_n}$ denotes the restriction of the
$\lambda$-coordinate function  to $\Sigma_n$, then the
following holds. If $P_+$ and $P_-$ are interior points
of $\Sigma_n$ where the function $\lambda_{\Sigma_n}$
obtains respectively its maximum and minimum value,
then the mean curvature vector at $P_+$ and $P_-$
points ``down'', toward $\wt\C^{\lambda_1}\cap\Omega_n$.

\begin{lemma}\label{slab}
Let $P^+$ (respectively $P^-$) be a point in
$\Sigma_n$ where the function $\lambda_{\Sigma_n}$ attains its
maximum (respectively minimum) value. Then, $P^+$
and $P^-$ cannot be in the interior of $\Sigma_n$
unless $\Sigma_n= \wt\C^{\lambda}\cap \Omega_n $ for a certain $\l \in (\l_1,\l_2)$.
\end{lemma}

\begin{proof}
By applying the maximum principle for constant mean
curvature surfaces, this lemma follows from the previous
observation and the fact that the collection
$$\wt \cF_n=\{\wt \C^{\lambda}\cap \Omega_n \mid  \l \in [\l_1,\l_2] \}$$
foliates $\Omega_n$.
\end{proof}

\subsection{Choosing the right boundary curve $\Gamma_n$}\label{boundary}
Let $\ov{\wt{W}}$ be the hyperbolic compactification of $\wt{W}$,
which is a covering  of the related
compactification of  $W$ when viewed to be a subset of $\ov{\HH}^3$.

For each $n\in \N$ large, we will construct a simple
closed curve $\Gamma_n$ in $\partial_* \Omega_n$ such that
the minimizer surface  $\Sigma_n \subset \Omega_n$
for the functional $I$ in~\eqref{Func-currents} with
$\partial \Sigma_n = \G_n$ is a $\wt{\theta}$-graph over its projection to
$[\l_1,\l_2 ]\times\{0\}\times  \R$.

Let $\Gamma_n$
be the union of four arcs in $\partial_* \Omega_n$,
\[
\Gamma_n:=\alpha^n_+\cup\beta^n_+\cup\alpha^n_-\cup\beta^n_-
\]
where $\alpha^n_+\subset \wt\C^{\lambda^+_n}\cap \wt \T_+^n$ with
$\lambda^+_n\nearrow \lambda_2$ and  with its endpoints on $\wt \Z_n^\pm$, and
$\alpha^n_-\subset \wt\C^{\lambda^-_n}\cap\wt \T_-^{-n}$
with $\lambda^-_n\searrow \lambda_1$
and with its endpoints on $\wt \Z_n^\pm$. The curve
$\beta^n_+\subset \wt\Z_n^+$ connects the endpoints of
$\alpha^n_+$ and $\alpha^n_-$ that are
contained in $\wt\Z_n^+$ while the curve
$\beta^n_-\subset \wt\Z_n^-$ connects the endpoints
of $\alpha^n_+$ and $\alpha^n_-$ that are
contained in $\wt\Z_n^-$; see Figure \ref{Gamma_n}.

\begin{figure}[t]
\begin{center}
$\begin{array}{c@{\hspace{.2in}}c}

\relabelbox  {\epsfxsize=2.5in \epsfbox{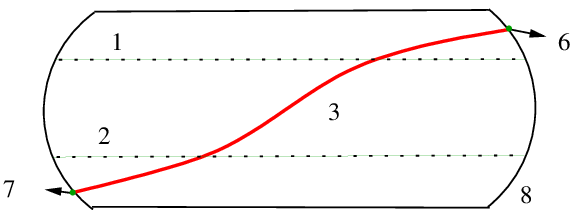}} \relabel{1}{\footnotesize $\wt{\C}^{\lambda_0^+}$} \relabel{2}{\footnotesize
$\wt{\C}^{\lambda_0^-}$} \relabel{3}{\footnotesize $\beta^n_+$}  \relabel{6}{\footnotesize $\alpha^n_+$} \relabel{7}{\footnotesize $\alpha^n_-$} \relabel{8}{\footnotesize $\wt{\Z}_n^+$}
\endrelabelbox &

\relabelbox  {\epsfxsize=2.5in \epsfbox{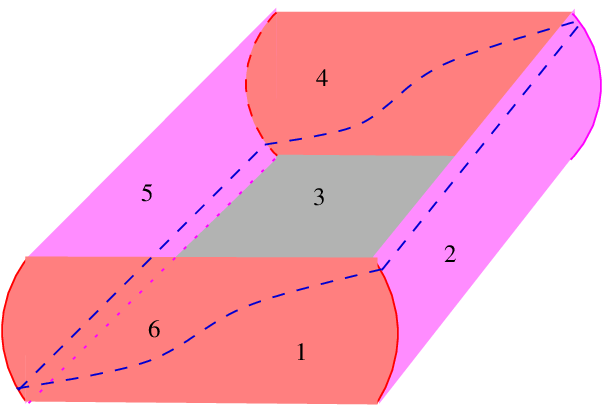}} \relabel{1}{\footnotesize $\wt{\Z}_n^+$} \relabel{2}{\footnotesize $T^n_+$}
\relabel{3}{\footnotesize $\Omega_n$} \relabel{4}{\footnotesize $\wt{\Z}_n^-$} \relabel{5}{\footnotesize $T^n_-$} \relabel{6}{\footnotesize $\Gamma_n$} \endrelabelbox \\ [0.4cm]

\end{array}$

\end{center}

\caption{\label{Gamma_n} In the left, $\beta^n_+$  is
pictured in $\wt{\Z}_n^+$. In the right, $\Gamma_n$ curve is described in
$\partial \Omega_n$.}

\end{figure}

Moreover, we chose $\beta^n_+$ and $\beta^n_-$ so that
they are smooth graphs 
in the $\wt \theta$-variable
with positive slope; see Figure~\ref{Gamma_n}.
We will assume that the curves $\beta^n_+$ and $\beta^n_-$ converge respectively
to a pair of curves $ \beta_+$, $ \beta_{-}$
in $\partial _\infty \wt{W}$.
With these choices, if we denote by $\widehat \beta^n_\pm$
the projections of $\beta^n_\pm$ in $\mathbb H^3\cap W_n$, the curves
$\widehat \beta^n_\pm$ are embedded curves contained in $ \Z_n^\pm$.
Finally, we require that $\widehat \beta^n_\pm$ to
converge to a pair of infinite smooth spiralling curves $\widehat \beta_\pm$
in  the pair of compact annuli $A^+,A^-$ in $\Si\cap \ov{W}$,
each of which is a graph of some smooth function $\pm\lambda( {\theta})$ with
positive slope and the graphs converge to the asymptotic boundary curves of $\C^{\l_1},\C^{\l_2}$.
Here $\ov{W}$ denotes the union of $W$ with its limit points in $ \Si$.
We will also assume that  $\widehat \beta_\pm$ have positive bounded geodesic curvature
and   the reflection in
the $(x,y)$-plane interchanges them; see Figure~\ref{spiral}.

We will next make  some further restrictions
on the choices of $\beta^n_+$ and $\beta^n_-$.
For each  $ p\in \widehat \beta_+$, let $C^1(p)$ and
$C^2(p)$ be the two circles in
$ \partial_\infty W\cap \{z>0\}\subset \Si$ of maximal radius such that
$C^1(p)\cap C^2(p)= p$ and such that
the pairwise disjoint open disks that they bound
in $\partial_\infty \HH^3$ are disjoint from $ \wh{\beta}_+$.
Note that these circles are on ``opposite''sides
of $\widehat \beta_+$ at $p$.
Denote by $\Delta^1(p)$ and $\Delta^2(p) $
rotationally symmetric open $H$-disks in $\HH^3$
with boundaries respectively $C^1(p)$ and $C^2(p)$; by the arguments in the
proof of Lemma~\ref{lem:3.1} below these disks must be disjoint from
 $\C^{\lambda_1}\cup \C^{\l_2}$. Furthermore, the disks are chosen so that
 the mean curvature
vector  of $\Delta^i(p)$  points into the component of
$W\setminus \cup_{i=1}^2 \Delta^i(p)$ that contains $\C^{\lambda_1}\cup \C^{\l_2}$, $i=1,2$.
The disks $\Delta^1(p)$ and $\Delta^2(p)$ can be defined in an
analogous way when $ p\in \widehat \beta_-$.

\begin{definition} \label{def:3.1} $ \displaystyle \Gamma=\widehat \beta_+\cup\widehat \beta_-$
and $ \displaystyle \wt{\Gamma}= \beta_+\cup \beta_-.$
\end{definition}

\begin{figure}[h]
\begin{center}
$\begin{array}{c@{\hspace{.2in}}c}

\relabelbox  {\epsfxsize=2.5in \epsfbox{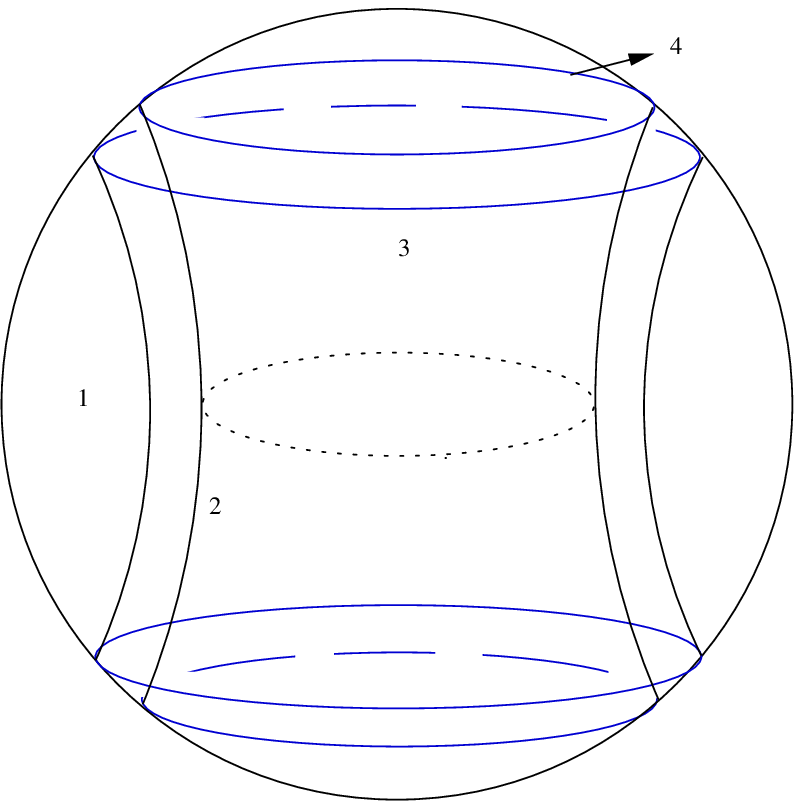}} \relabel{1}{\footnotesize $\C^{\l_1}$} \relabel{2}{\footnotesize $\C^{\l_2}$}
\relabel{3}{\footnotesize $\gamma_1^+$}  \relabel{4}{\footnotesize $\gamma_2^+$}  \endrelabelbox &

\relabelbox  {\epsfxsize=2.5in \epsfbox{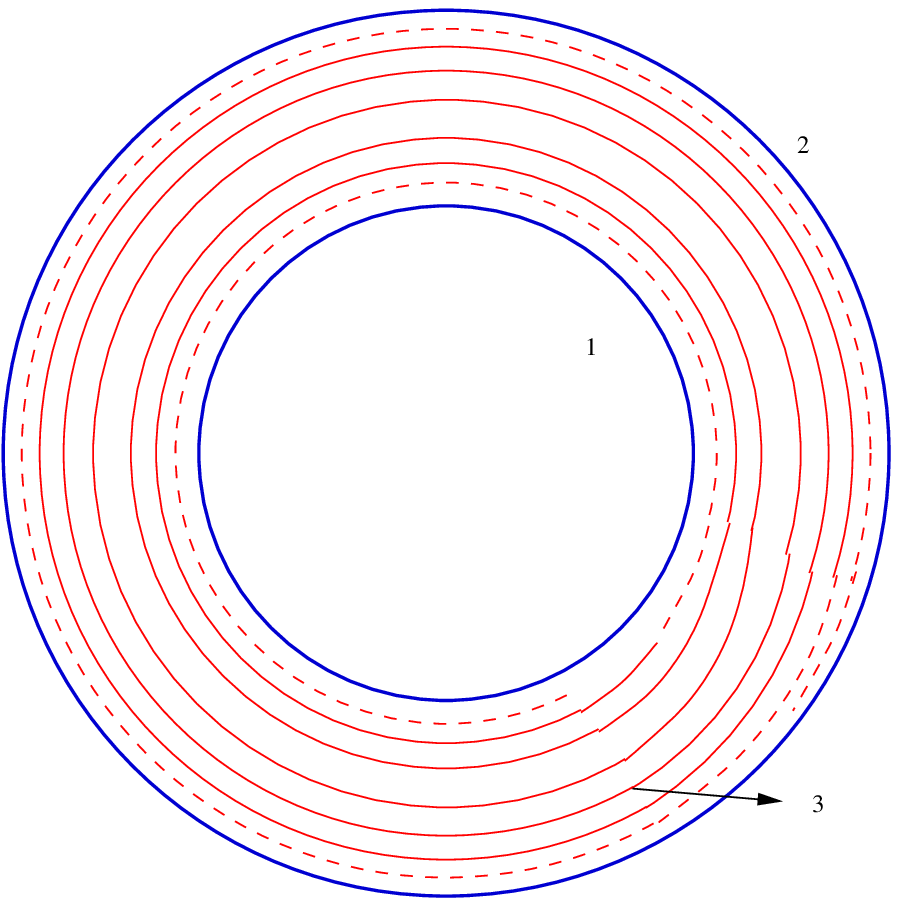}} \relabel{1}{\footnotesize $\gamma_1^+$} \relabel{2}{\footnotesize $\gamma_2^+$} \relabel{3}{\footnotesize $\Gamma^+$}
 \endrelabelbox \\
[0.4cm]
\end{array}$

\end{center}
\caption{\label{spiral} Let $\PI \C^{\l_i}=\gamma_i^+\cup\gamma_i^-$, $i=1,2$, and
$\Gamma=\Gamma^+\cup \Gamma^-$. Then, $\Gamma^+$ is an infinite line
in $\Si$ spiraling into $\gamma^+_1$ in one end, and spiraling into $\gamma^+_2$ in the other end. }
\end{figure}

Abusing the notation,  for each $p\in \wt{\G}$,
we let $\Delta^1(p),\Delta^2(p)$ denotes the lifts at $p$ of the related
disks $\Delta^1(\Pi(p))$, $\Delta^2(\Pi(p))$.
The final condition on the convergence
of $\beta^n_+$ and $\beta^n_-$ to $\beta_+$ and $\beta_-$ is that
\[
\beta^n_+\cap \bigcup _{p\in\beta_+}[\Delta^1(p)\cup\Delta^2(p)]= \mbox{\O},
\quad \beta^n_-\cap \bigcup _{p\in\beta_-}[\Delta^1(p)\cup\Delta^2(p)]=\mbox{\O},
\]
for all $n$ sufficiently large.

Necessarily, if $\lambda_{\beta_\pm}$ denotes the
 restriction of the coordinate function $\lambda$ to
 $\beta_\pm$ then $\lambda_{\beta_\pm}(\wt{\theta})\to \lambda_2$ as $\wt{\theta}\to +\infty$ and
$\lambda_{\beta_\pm}(\wt{\theta})\to \lambda_1$ as $\wt{\theta}\to -\infty$, which means
that they are spiralling toward
$[\partial \C^{\lambda_1}\cup \partial \C^{\lambda_2}]\subset \Si $, e.g., $\lambda(\wt{\theta})
= \lambda_1+ (\lambda_2-\lambda_1)\dfrac{2\arctan{\wt{\theta}}+\pi}{2\pi}$.

\begin{lemma} \label{lem:3.1}
For all $p\in \G$ and all $n$ sufficiently large,
the compact surfaces $\Pi(\Sigma_n)$ are disjoint from
$\Delta^1(p)\cup \Delta^2(p) $.  In particular, for $n$ sufficiently  large,
the compact surfaces $\Pi(\Sigma_n)$ are disjoint from
 the two components of
 $\HH^3 \setminus [\Delta^1(p)\cup \Delta^2(p)] $ that do not contain the origin.
Furthermore,  the limit set of the closed set
 $\ov{\cup_{n=1}^\infty \Pi(\Sigma_n)}\subset \HH^3$ in $\partial_\infty \HH^3$
 must be  contained  in the closed set $\ov{\G}\subset \partial_\infty \HH^3$.
 \end{lemma}
\begin{proof}
The proof will  follow from a simple application of the maximum principle.
Fix $p\in \G$. By the construction of the curves
$\G_n=\partial \Sigma_n$, for $n$ sufficiently  large,
$\Pi(\partial \Sigma_n)\cap [\Delta^1(p)\cup \Delta^2(p)]=\mbox{\O}$.
If $\Pi(\Sigma_n)\cap [\Delta^1(p)\cup \Delta^2(p)]\neq \mbox{\O}$,
then one of the two disks,
say $\Delta^1(p)$, intersects $\Pi(\Sigma_n)$
 at some point.  Note that the closure $B$
of the component of $\HH^3 -\Delta^1(p)$ that
 is disjoint from the origin is
foliated by rotationally symmetric open disks $D(t)$ of  constant
mean curvature $H$ each of which is properly embedded in $\HH^3$,
where $t\in [0,\infty)$ and  $D(0)=\Delta^1(p).$
Since for every $t\geq 0$ and $n$ sufficiently large, $ D(t)$ is
disjoint from $\Pi(\partial \Sigma_{n})$ and $\Pi( \Sigma_{n})$ is compact,
there exists a largest non-negative number $t_0$ such that
$D(t_0)$ intersects $\Pi(\Sigma_n)$. Therefore $D(t_0)$
locally lies on one side of $\Pi(\Sigma_n)$
at an interior point of intersection. This contradicts the maximum principle and proves the
first statement in the lemma.
The second statement of Lemma~\ref{lem:3.1}
follows immediately from the first statement.

After viewing $\ov{\HH}^3$ with the closed unit ball metric,
let $q\in \partial_\infty \HH^3 \setminus
([\partial_\infty \C^{\l_1}\cup \partial_\infty \C^{\l_2} \cup \G]=\ov{\G})$.
By construction, the distance from
$q$ to $\cup_{i=1}^\infty \Pi(\partial \Sigma_n)$ is positive
in the closed ball metric on $\ov{\HH}^3$.
The arguments in the first paragraph of this proof using the maximum
principle  show  that there exists a disk $D^q\subset \HH^3$
of revolution and constant mean curvature $H$, with boundary circle
in $\partial_\infty \HH^3$ centered at $q$, that is disjoint from
$\cup_{n=1}^\infty \Sigma_n$.  The existence of $D^q$ implies that
$q$ is not in  the closure of $\cup_{n=1}^\infty \Pi(\Sigma_n)$
in   $\ov{\HH}^3$.  This completes the proof of the lemma.
\end{proof}


\section{Constructing the surface $\Sigma_H$} \label{sec4} In this section we construct
$\Sigma_H$ and finish the proof of Theorem~\ref{main}.
\begin{lemma} \label{translation}
Let $\Gamma_n$ be as described in Section~\ref{boundary} and let
$E_n=\Omega_n\cap([\l_n^{-},\l_n^+]\times\{0\}\times   \R )$. Then  $\Sigma_n$ is a
 $\wt{\theta}$-graph over the compact disk $E_n$. In particular, the related
 Jacobi function $J_n$ on $\Sigma_n$ induced by the inner product of
 the unit normal field to $\Sigma_n$ with the Killing field
 $ \partial_{\wt{\theta}}$
 is positive in the interior of $\Sigma_n$.
\end{lemma}

\begin{proof} Recall that $T_\alpha$ is an isometry
of $\wt W$, which is
translation by $\alpha$, i.e.,
$ T_{\alpha}(\lambda, \wt{\theta}, z)=(\lambda, \wt{\theta}+\alpha, z)$.
Let $T_\alpha(\Sigma_n)=\Sigma^\alpha_n$ and
$T_\alpha(\Gamma_n)=\Gamma^\alpha_n$. We claim that
$\Sigma^\alpha_n\cap\Sigma_n=\mbox{\rm \O} $
for any $\alpha \in \mathbb{R}\setminus\{0\}$ which implies that $\Sigma_n$ is a
 $\wt{\theta}$-graph.

Arguing by contradiction, suppose that $\Sigma^\alpha_n\cap\Sigma_n\neq\mbox{\rm \O} $  for a certain
$\alpha\neq 0$.   By compactness of $\Sigma_n$,
there exists a largest positive number $\alpha'$ such that
$\Sigma^{\alpha'}_n\cap\Sigma_n\neq \mbox{\rm \O} $.
Let $p\in \Sigma^{\alpha'}_n\cap\Sigma_n$. Since
$\partial \Sigma^{\alpha'}_n \cap \partial \Sigma_n =\mbox{\rm \O} $
and, by Lemma~\ref{slab}, the interiors
of both $\Sigma^{\alpha'}_n$ and $\Sigma_n$ lie in
$(\l_n^{-},\l_n^+)\times\R \times   \R $, then
$p\in \Int(\Sigma^{\alpha'}_n) \cap \Int(\Sigma_n)$.
Since the surfaces $\Int(\Sigma^{\alpha'}_n)$,
$\Int(\Sigma_n)$ lie on one side of each other
and intersect tangentially at the point $p$ with
the same mean curvature vector,
then we obtain a contradiction to the maximum
principle for constant mean curvature surfaces. This proves that $\Sigma_n$ is
graphical over its $\wt \theta$-projection to
$E_n$.

Since by construction every integral curve,
$(\overline \l,t,\overline z)$ with $\overline \l,\overline z$
fixed and  $(\overline \l,0, \overline z)\in E_n$, of the Killing field $\partial _{\wt{\theta}}$ has
non-zero intersection number with any compact
surface  bounded by $\G_n$, we conclude that
every such integral curve intersects both the
disk $E_n$ and $\Sigma_n$ in single points.
This means that   $\Sigma_n$ is a $\wt{\theta}$-graphs over
$E_n$ and thus the related Jacobi function
$J_n$ on $\Sigma_n$ induced by the inner product of
 the unit normal field to $\Sigma_n$ with the
 Killing field $\partial _{\wt{\theta}}$
 is non-negative in the interior of $\Sigma_n$. Since $J_n$ is a
 non-negative Jacobi function, then either
 $J_n\equiv 0$ or $J_n>0$. Since $J_n$ is positive somewhere
 in the interior, then $J_n$ is positive everywhere in the interior.
 This finishes the proof of the lemma.
\end{proof}

To summarize, with $\Gamma_n$ as previously
described, we have constructed a sequence
of  compact stable $H$-disks $\Sigma_n$
with $\partial \Sigma_n = \Gamma_n \subset \partial \Omega_n$.
%
%
By the curvature estimates for stable $H$-surfaces given  in~\cite{rst1},
the norms of the second fundamental forms of the $\Pi(\Sigma_n)$ are uniformly
bounded from above at points at least $\ve>0$ intrinsically far
from their boundaries, for any $\ve>0$. Since for any compact subset $X\subset \Int(W)$
and for $n$ sufficiently large,
$\Gamma_n$ is a positive distance from $X$, the norms of the second fundamental
forms of the $\Pi(\Sigma_n)$ are uniformly bounded on compact
sets of $ \Int(W)$.

A standard compactness
argument, using the uniform curvature estimates for the surfaces $\Sigma_n$
on compact subsets of $\Int(W)$ and
their graphical nature described in Lemma~\ref{translation},  implies that
a subsequence $\Pi(\Sigma_{n(k)})$ of the surfaces $\Pi(\Sigma_n)$ converges to
an $H$-lamination ${\cL}$ of $\Int(W)$. By the nature of the convergence,
each leaf of $\cL$ has bounded norm of its second fundamental form  on compact
sets of $\Int(W)$ and this bound depends on the compact set but not on the leaf.

By similar arguments, there exists an $H$-lamination  $\wt{\cL}$ of $\wt{W}$
that is a limit of some subsequence of  the graphs $\Sigma_{n(k)}$; note that
in this case we include the boundary of $\wt{W}$ in the domain, contrary to
the previous case where we constructed the lamination $\cL$ in the interior of $W$.
After a refinement of the original subsequence, we will assume that
$\Pi(\Sigma_{n(k)})$ converges to $\mathcal L$ and that
 $\Sigma_{n(k)}$ converges to $\wt{\cL}$. Since the boundaries of the
 $\Sigma_{n(k)}$ leave every compact subset of $\wt{W}$, the leaves
of $\wt{\cL}$ are complete. Note also that the leaves of $\wt{\cL}$ have uniformly bounded norms of their second fundamental forms. The fact that the laminations $\cL$ and $\wt{\cL}$ are not empty
will follow from the next discussion.

Note that the region $\Int(W)$ is foliated by the integral
curves of the Killing field $\partial_\theta$, which are circles and
each such circle intersects  $ B=(\l_1,\l_2)\times \{0\}\times \R$ orthogonally  in a unique point $b$;
let $S(b)$ denote this circle. We next study how the circles $S(b)$
intersect  the leaves of $\cL$ and prove some properties of the laminations
$\cL$ and $\wt{\cL}$. Let $\Theta \colon \Int(W) \to B$ denote the natural projection.

\begin{claim}\label{claimone}
For every $b\in B$, $S(b)$ intersects at least one of the leaves of $\cL$.
Furthermore, if $S(b)$ intersects a leaf $L$ of $\cL$ transversely at some point $p$,
then $L$ is the only leaf of $\cL$ that intersects $S(b)$ and $S(b)\cap L=\{p\}$.
\end{claim}
\begin{proof}
The first statement in this claim follows from the fact that for any $b\in B$,
for $n$ sufficiently large, $S(b)$ intersects
the ``graphical" surface $\Pi(\Sigma_{n(k)})$
in a single point $p_{n(k)}$, and since $S(b)$ is compact, some subsequence of
these points converges to a point $p\in S(b)\cap\cL$.

If $S(b)$ intersects a leaf $L$ of $\cL$ transversely
in a point $q$, then there exists an $\ve(q)>0$ such that
for $n$ sufficiently large, a small neighborhood
$N(q,L)\subset L$ of $q$
is a $\theta$-graph of bounded gradient over the disk $D_B(b,\ve(b))\subset B$
centered at $b$ of radius $\ve(b)$,
and $\{\Theta^{-1}(D_B(b,\ve(b)))\cap \Sigma_{n(k)}\}_{n(k)}$
is a sequence of graphs  converging smoothly
to the graph $N(q,L)=\Theta^{-1}(D_B(b,\ve(b)))\cap \Sigma_{n(k)}$.
In particular, $S(b)$ intersects $\cL$ transversely
in the single point $q$.
\end{proof}

 \begin{claim}\label{claimtwo}
 The limit set $\partial_\infty \cL$ contains $\G$ and it contains no other points in
the two annuli in $\partial_\infty W \setminus [\partial_\infty \C^{\l_1} \cup \partial_\infty \C^{\l_2}]$.
\end{claim}
\begin{proof}
Let Lim$ (\cL)$  denote the limit set of $\cL$
that lies in $\partial_\infty W \setminus [\partial_\infty \C^{\l_1} \cup \partial_\infty \C^{\l_2}]$.
By Lemma~\ref{lem:3.1},  Lim$( \cL)\subset \G$. We next show that $\G\subset \mbox{\rm Lim}( \cL)$.
Let $x\in  \G$ and choose a sequence of circles $\{S(b_k)\}_{k\in \N}$
that converges to the circle $C(x)\subset \partial_\infty \HH^3$ passing through $x$.
By Claim~\ref{claimone}, there exist points $p_k\in S(b_k)\cap \cL$, and by compactness of $\ov{\HH}^3$, a
subsequence of these points converges to a point $p\in C(x)\cap \mbox{\rm Lim}(\cL)$.
But since $C(x)\cap \G=\{x\}$, then $x=p$, which completes the proof that $\mbox{\rm  Lim}( \cL)= \G,$
and the claim holds.
\end{proof}

\begin{claim} \label{claim:intersects}
The limit set of $\wt{\cL}$ in $\partial_\infty \wt{W}$ is equal to $\wt{\G}$.
\end{claim}
\begin{proof}
By arguing with barriers as in the proof of Lemma~\ref{lem:3.1},  the limit set of
 $\wt{\cL}$  must be contained in  $\wt{\G}$.
 Let $p\in \beta^+ \subset \partial_\infty \wt{W}$ and let  $\Delta^1(p), \Delta^2(p)$  be the disks
described at the end of the previous section. Consider
a small arc $\alpha \subset \ov{\wt{W}}$ with end points
in $\Delta^1(p)\cup \Delta^2(p)$ that links $\beta^+$
and such that $\Pi(\alpha)$ is  the compactification
of a geodesic in $\wt{W}$. Then by previous arguments, $\alpha$ must intersect a
leaf $L$ of $\wt{\cL}$.  Since a sequence of these arcs can be chosen to converge
to $p$, then $p$ is in the limit set of  $\wt{\cL}$.

Using exactly the same arguments gives that the same is true of $p\in\beta^-$.
This completes the proof of the claim.
\end{proof}

\begin{remark}\label{remark}{\em
Note that this claim  implies that no leaves of  $\wt{\cL}$ are invariant under the one-parameter
group of translations $T_{\wt{\theta}}$, since such complete
surfaces are the lifts of surfaces of revolution in $\HH^3$,
and as such they  have their limit sets that contain circles which are not
contained in $\wt{\G}$. In particular, $\wt{\cL}$ does not contain
$\widetilde{\cC}^{\l_1}$ nor $\widetilde{\cC}^{\l_2}$.}
\end{remark}

\begin{claim} \label{compact-arc}
Let $\alpha\subset \ov{W}\subset [\HH^3\cup\Si]$ be a compact arc with
$\Int(\alpha) \subset W$, joining
$\C^{\lambda_1}\cup\partial_\infty \C^{\lambda_1}$ to
$\C^{\lambda_2}\cup\partial_\infty \C^{\lambda_2}$. Then there exists a leaf $L$ of
$\cL$
that  intersects $\alpha$ and that is not invariant under rotations around the $z$-axis.
\end{claim}

\begin{proof}
Let $\wt{\alpha}\subset \ov{\wt{W}}$ be any fixed lift of $\alpha$.
 By a linking argument, it follows that when $n(k)$ is sufficiently
large,  $\wt{\alpha}$ intersect $\Sigma_{n(k)}$  at some interior point $p_{n(k)}$.
Suppose that a subsequence of the points $p_{n(k)}$ converges to an end point of
$\wt{\alpha}$ that corresponds to an end point of $\alpha$ in
$\C^{\lambda_1}\cup\C^{\lambda_2}$. 
This would imply   that $\widetilde{\cC}^{\l_1}$ or
$\widetilde{\cC}^{\l_2}$ is a leaf of the lamination $\wt{\cL}$, contradicting the previous remark.
Next, suppose that a subsequence of the points $p_{n(k)}$ converges to an end point of
$\wt{\alpha}$ that corresponds to an end point of $\alpha$ in
$ \partial_\infty{\C}^{\lambda_1}\cup\partial_\infty {\C^{\lambda_2}}$.
This picture is ruled out by Claim~\ref{claim:intersects}.
Therefore a subsequence of the points  $p_{n(k)}$  must converge to a point
in the interior of $\wt{\alpha}$, which is
therefore a point on
some leaf $\wt{L}$ of $\wt{\cL}$. Note that $\alpha$ intersects $L=\Pi(\wt{L})$
and, since  $\wt{L}$ is not invariant under the one-parameter
group of translations $T_{\wt{\theta}}$,  $L$ is not invariant under rotations
around the $z$-axis
Thus the claim holds.
\end{proof}

\begin{claim}\label{complete}
Every complete leaf $L$ of $\cL$ is the graph of a smooth function
defined on its $\theta$-projection  $\Theta(L)\subset B$.
\end{claim}
\begin{proof}
 Let $L$ be a complete leaf of $\cL$.
Recall that  the surfaces $\Pi(\Sigma_{n(k)})$ are
$ {\theta}$-graphs
and let $J_{n(k)}$ denote the related positive Jacobi functions
induced by the inner product of
the unit normal field to $\Pi(\Sigma_{n(k)})$ with the Killing
field $ \partial_{\theta }$.
Let $J$ denote the limit Jacobi function on the leaves of $\cL$
and let $J_L$ be the related Jacobi function on $L$.
Since $L $ is the limit of portions of the surfaces
$\Pi(\Sigma_{n(k)})$, the previous observation
implies that  $J_L$
is non-negative. Since $J_L$ is a
non-negative Jacobi function, then either
$J_L\equiv 0$ or $J_L>0$. If $J_L>0$,
then by Claim~\ref{claimone}, $L$ is  the  graph of a smooth function over its projection to
$(\l_1,\l_2)\times \{0\}\times  \R$. Therefore, to prove the claim,  it suffices to show that $J_L>0$.

Arguing by contradiction, assume that $J_L\equiv 0$,
then $L $ is a complete embedded surface in $W$ that is invariant under rotations
around the $z$-axis. In particular,  there exists a complete arc $\beta$ in $L\cap E$ and $\beta $ is embedded 
since $L$ is embedded; completeness of $\beta$ follows from the completeness of $L$. 
Also this arc cannot be bounded in $\HH^3$ since otherwise $\ov{L}$ would be
a bounded $H$-lamination in $\HH^3$, which  is impossible since there would exist a leaf $L'$ of  $\ov{L}$ 
that would be contained in  a geodesic ball $B_{\HH^3}$ centered at the origin and tangent to $\partial B_{\HH^3}$ at
some point $q$.
But   $\partial B_{\HH^3}$ has mean curvature greater than one, which is a contradiction to the mean curvature
comparison principle applied at the point $q$.  Hence, since $\beta$ is not
bounded in $\HH^3$, then $\partial_\infty L \neq\mbox{\rm \O}$.

Since $L $ is  invariant under rotation
around the $z$-axis and $\partial_\infty \cL\subset \Gamma\cup \partial_\infty \C^{\l_1} \cup \partial_\infty \C^{\l_2} $, then
$\partial_\infty L \subset \partial_\infty \C^{\l_1} \cup \partial_\infty \C^{\l_2} $.
To obtain a contradiction we will consider separately
each of the following three possible cases for the limiting behavior of $L $. \vspace{.2cm}

\noindent{\bf Case A:} \, $\partial_\infty L  \subset \partial_\infty \C^{\l_1}$ or
$\partial_\infty L  \subset   \partial_\infty \C^{\l_2} $
and $L $ is a spherical catenoid
as described in the Appendix with its mean curvature vector  pointing toward the $z$-axis.

\noindent{\bf Case B:}  \, $\partial_\infty L  $ contains one component in $\partial_\infty \C^{\l_1} $
and another   component in $\partial_\infty \C^{\l_2} $.

\noindent{\bf Case C:}  \, $\partial_\infty L  \subset \partial_\infty \C^{\l_1}$ or
$\partial_\infty L \subset   \partial_\infty \C^{\l_2} $
and $L $ is a surface of revolution
as described in Gomes~\cite{gom1} with its mean curvature vector  pointing away from the $z$-axis.
Note that in this case it might hold that  $\partial_\infty L $ is a single circle with multiplicity two.
\vspace{.2cm}

First suppose that Case A holds. By the discussion in the Appendix and the
description of stable catenoids, the only possible spherical catenoids
are $\C^{\lambda_1}$ or $\C^{\lambda_2}$ but they do not intersect $\Int (W)$, which
gives a contradiction.

If Case B holds, then there is a compact arc $\alpha\subset L \cup\partial_\infty L $
satisfying the hypotheses of Claim~\ref{compact-arc}.  By the same claim, there must exist
a non-rotational leaf $L_1$ of $\cL$  that intersects $L$, which is impossible since
distinct leaves of $\cL$ are disjoint.

Finally suppose that Case C holds.
If $H=0$, then using the maximum principle applied to the foliation
of $W$ by minimal catenoids gives an immediate contradiction;
hence, assume that $H>0$.
As remarked in Gomes~\cite{gom1}, the properly embedded surfaces of
revolution in $\HH^3$ extend to $C^1$-immersed surfaces in $\ov{\HH}^3$ and they
make a particular
oriented angle $\theta_H\neq\pi/2$ with $\partial_\infty \HH^3$ according to their orientation.
First consider the case where $L $ contains a single circle $S$ component in
$\partial_\infty \C^{\lambda_1}\cup \partial_\infty \C^{\lambda_2}$.
This means that
the surface $L $ makes two different positive oriented angles along its
single boundary circle $S$ at infinity. Since $L$ lies on the side of
the spherical catenoid  $\C^{\lambda_2}$ that makes an acute angle with $\partial_\infty \HH^3$,
and  $\C^{\lambda_2}$ is disjoint from $L$, if
$S\subset \partial_\infty C^{\lambda_2}$, we obtain a contradiction.
Hence, we may assume that $S$ is a boundary curve of $\C^{\lambda_1}$.
But in this case, there is a largest $\lambda\in [\lambda_2,\lambda_1)$ such that $\C^\lambda$ intersects
$L$ at an interior point and $L$ lies on the mean convex side
of $\C^\lambda$. This contradicts the maximum principle.
Therefore, if Case C holds, then either $\partial_\infty L  = \partial_\infty \C^{\l_1}$ or
$\partial_\infty L  =   \partial_\infty \C^{\l_2} $.

Reasoning as in the previous paragraph with the angles along the boundary, we find that
the only possibility is $\partial_\infty L  = \partial_\infty \C^{\l_1}$.
But in this case,   there is again a largest
$\lambda\in [\lambda_2,\lambda_1)$ such that $\C^\lambda$ intersects
$L$ at an interior point and $L$ lies on the mean
convex side of $\C^\lambda$. This contradicts the maximum principle.
This final contradiction completes the proof of the claim.
\end{proof}
Consider a complete leaf $L$ of $\wt L$.  Recall that $L$ has bounded
norm of the second fundamental form, is not invariant under the one-parameter
group of translations $T_{\wt{\theta}}$, and by the maximum principle  $L\subset \Int(\wt{W})$.
It follows from the construction of the $H$-laminations $\cL$ and $\wt{\cL}$ that
$\Sigma_H=\Pi(L)$ is a complete leaf of $\cL$. By the previous
lemma, $\Sigma_H$ is the graph of a smooth function
defined on its $\theta$-projection  $\Theta(\Sigma_H)\subset B$.

We next prove that the leaf $\Sigma_H$ is properly embedded in $\Int(W)$.
If not, then there exists a limit leaf $L\neq \Sigma_H$
of $\cL$ in the closure of $\Sigma_H$.  Since this
leaf is easily seen to be complete as well, by Claim~\ref{complete} $L$
 is a graph of a smooth function  over its projection $\Theta(L)\subset B$.
 Since the open sets $\Theta(\Sigma_H), \Theta(L)$ intersect near
any point of $\Theta(L)$, this contradicts Claim~\ref{claimone} and so $\Sigma_H$
is properly embedded in $\Int(W)$. Since $\Theta \colon \Int(W) \to B$ is a
proper submersion, $\Theta(\Sigma_H)=B$ and thus, by Claim~\ref{claimone},
$\Sigma_H$ is the unique leaf of $\cL$.
Clearly the closure $\ov{\Sigma}_H$ of ${\Sigma}_H$ in $W$ is $ \Sigma_H\cup\C^{\l_1}\cup  \C^{\l_2}$.

To summarize, we have shown that $\Sigma_H$ is a complete graph over $B$
and    is properly embedded in $\Int(W)$.
Moreover: \ben
\item $\Sigma_H$ has bounded norm of the second
fundamental form.
\item The closure $\ov{\Sigma}_H$ of ${\Sigma}_H$ in $W$ is $ \Sigma_H\cup\C^{\l_1}\cup  \C^{\l_2}$.
\item $\partial_\infty \Sigma_H=\Gamma\cup \partial_\infty \C^{\l_1}\cup \partial_\infty \C^{\l_2}$;
see Figure \ref{spiral}.
\een
In particular, the surfaces $\{T_\theta(\Sigma_H)\mid \theta \in [0,2\pi)\}$ together with
the spherical catenoids $ \C^{\l_1}, \C^{\l_2}$ form an $H$-foliation of $W$.
This finishes the proof of the Theorem~\ref{main}.

\section{Appendix}

In this appendix, we  recall some facts about
spherical $H$-catenoids in $\HH^3$ that are used throughout the paper.
 As we have done in the previous sections,
we will work in $\mathbb{H}^3$ using the Poincaré ball model, that is
we consider $\mathbb{H}^3$ as the unit ball in $\rth$ and   its
ideal boundary at infinity corresponds to the boundary of the ball.
We will let $P_{x,y}$, $P_{x,z}$ and $P_{y,z}$ be the related totally
geodesic coordinate planes
in this ball model of $\HH^3$.

\begin{definition} Let $\C$ be a properly
immersed  annulus in $\mathbb{H}^3$ with
constant mean curvature $H\in [0,1]$ and $\PI \C=\alpha_1\cup\alpha_2\subset \Si$
where $\alpha_1$ and $\alpha_2$ are two disjoint round circles in
$\Si$. Let $\gamma$ be the unique geodesic in
$\mathbb{H}^3$ from the center of $\alpha_1$ to the
center of $\alpha_2$. If $\C$ is rotationally invariant with
respect to $\gamma$,
then we call $\C$ a {\em spherical $H$-catenoid} with rotation axis $\gamma$.
\end{definition}

In other words, spherical $H$-catenoids are obtained by rotating
a certain curve, a catenary, around a geodesic. In~\cite{gom1}, Gomes studied
the spherical $H$-catenoids in $\mathbb{H}^3$ for
$0\leq H\leq1$, and classified them in terms
of the generating curve   which is a solution to a
certain ODE. In what follows, when $H\neq 0$, we will
only focus on the spherical $H$-catenoids for which
the mean curvature vector points towards the rotation
axis.

After applying an isometry,  we can assume that the
$x$-axis is the rotation axis for the spherical
$H$-catenoid and that the circles $\alpha_1$ and $\alpha_2$
are symmetric with respect to the $(y,z)$-plane.
Gomes proved that for a fixed $H\in [0,1)$ and
$\l\in(0,\infty)$ there exists a unique spherical
$H$-catenoid $\C_{H}^\lambda$ with generating curve
$\beta_H^\lambda$ and satisfying the following properties:
\begin{enumerate}
\item $\beta_H^\lambda$ is the graph in
normal coordinates  over the $x$-axis of a positive
even function $g^\lambda_H(x)$ defined on an interval
$(-d^\lambda_H, d^\lambda_H)$, namely
$$\beta_H^\lambda= \{(x,g^\lambda_H(x)) \mid x\in(-d^\lambda_H, d^\lambda_H)\},$$ and satisfying
$\lambda=g^\lambda_H(0)$; see Figure \ref{gomesfigures}-right;
\item Up to isometry,
a spherical $H$-catenoid is isometric to a certain
$\C_{H}^\lambda$.
\end{enumerate}
In particular, $\C_H^\lambda$ is embedded and symmetric
with respect to the $(y,z)$-plane. Recall that if $\Sigma$ is a
properly embedded $H$-surface, $H<1$ in $\HH^3$, with limit set at infinity
a smooth compact embedded curve with multiplicity one, then the closure of $\Sigma$ in the closed ball
$\overline{\HH}^3$ is a $C^1$ surface that makes a constant angle with
the boundary sphere and this angle is the same as the one
that a complete curve in $\HH^2$ of constant geodesic curvature $H$ makes with the
circle $\partial_\infty \HH^2$.

Gomes also analyzed the relation between the
spherical $H$-catenoid $\C^\lambda_H$ and its
asymptotic boundary $\PI \C^\lambda_H = \tau_{H,\lambda}^+
\cup \tau_{H,\lambda}^-$ where each of the curves
$\tau_{H,\lambda}^\pm$ is a circle in $\Si$. Define the
asymptotic distance function $d_H(\lambda)$ as the
distance between the geodesic planes $P^+_{H,\lambda}$
and $P^-_{H,\lambda}$ where
$\PI P^\pm_{H,\lambda}= \tau_{H,\lambda}^\pm$, respectively.
By an abuse of language, we will say that
$d_H(\lambda)$  is the distance between the asymptotic
circles of $\C^\lambda_{H}$. Recall that we use
normal coordinates  over the $x$-axis in the description of the function
$g_H^\lambda(x)$.  Let the asymptotic boundary
of the generating curve $\beta^\lambda_H$ be
the points $\{w_{H,\lambda}^+,w_{H,\lambda}^-\}\subset \Si$, i.e.,
$\PI \beta^\lambda_H = \{w_{H,\lambda}^+,w_{H,\lambda}^-\}\subset \Si$. Then the
geodesic projections of these points to the $x$-axis will be
$\Pi(w_{H,\lambda}^\pm)=\pm d^\lambda_H$. Since $w_{H,\lambda}^\pm \subset
\tau_{H,\lambda}^\pm$ and the construction is rotationally invariant,
it is easy to see that $d_H(\lambda) = 2d^\lambda_H$.

Gomes showed that for a fixed $H\in[0,1)$, the function $d_{H}(\lambda)$ increases
from a non-negative value $\lim_{\lambda\to0}d_{H}(\lambda)$, which is
zero and not acquired when $H=0$ and positive when $H>0$, reaches a maximum at a certain $c_H\in (0,\infty)$,
and then  decreases to $0$ as $\lambda\to \infty$; see Figure~\ref{gomesfigures} and \cite[Lemma 3.5]{gom1}.

\begin{figure}[h]
\begin{center}
$\begin{array}{c@{\hspace{.2in}}c}

\relabelbox  {\epsfxsize=2.5in \epsfbox{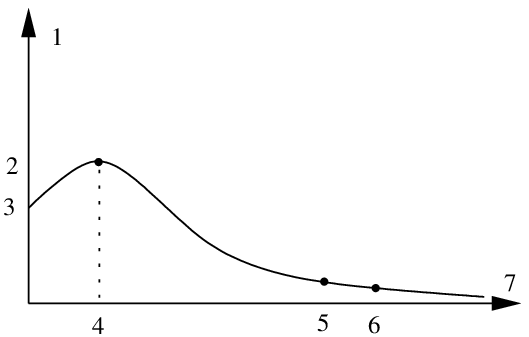}} \relabel{1}{\footnotesize $d_H(\lambda)$} \relabel{4}{\footnotesize $c_H$}
\relabel{5}{\footnotesize $\lambda_1$} \relabel{6}{\footnotesize $\lambda_2$} \relabel{7}{\footnotesize $\lambda$} \endrelabelbox &

\relabelbox  {\epsfxsize=2.5in \epsfbox{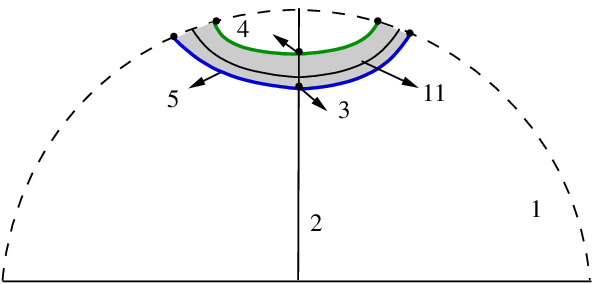}} \relabel{1}{\scriptsize $S_\infty$} \relabel{2}{$\lambda$} \relabel{3}{\scriptsize $\lambda_1$}
\relabel{4}{\scriptsize $\lambda_2$} \relabel{5}{\scriptsize $\beta_{\lambda_1}$} \relabel{11}{\scriptsize $\beta_\lambda$}  \endrelabelbox \\
[0.4cm]
\end{array}$

\end{center}
\caption{\label{gomesfigures} $\lambda$ represents the distance from the rotation axis,
and $d_H(\lambda)$ represents the asymptotic distance between
the asymptotic circles of $\C^\lambda_H$.}
\end{figure}

In particular, this discussion implies that for a fixed $H\in[0,1)$, the distance between the
asymptotic boundaries of the spherical $H$-catenoids is bounded by $\max(d_{H})$.
He also showed that $\max({d_H})$ is monotone increasing in $H$, and
$\max({d_H}) \to \infty$ as $H\to 1$. The following lemma summarizes some of
these results.

\begin{lemma}[\cite{gom1}] \label{Hcatenoid}  For any $H\in [0,1)$ and any $\l\in(0,\infty)$,
there exists  a unique  spherical $H$-catenoid $\C^\lambda_{H}$ in $\mathbb{H}^3$ such
that the distance from its generating curve $\beta^\lambda_{H}$ to the $x$-axis is
$\lambda$. Moreover, there exists a number
  $c_H>0$ such that $d_{H}(c_H)=\max_{(0,\infty)}d_{H}$, and $d_H(\lambda)$ decreases
  to $0$ on the interval  $[c_{H},\infty)$ as $\lambda$ goes to infinity; see Figure \ref{gomesfigures}-left.
\end{lemma}

In the next proposition we  construct the foliations by spherical $H$-catenoids
that are used in producing our examples.

\begin{proposition}~\label{foliation} For any $ H\in [0, 1)$,
the family of spherical $H$-catenoids
$\cF_H=\{ \C^\lambda_H \ | \ \lambda\in [c_H,\infty)\}$  foliates
the  closure of the non-simply-connected component of \, $\mathbb{H}^3\setminus \C^{c_H}_H$.
In particular, each of the spherical catenoids in $\cF_H$ admits a positive Jacobi function,
that is induced by the associated normal variational field.
\end{proposition}

\begin{proof} We first consider the case $H\in (0,1)$. After adding the leaf
$\C^{c_H}_H$,
it suffices to  show that
the family of spherical $H$-catenoids
$\cF_H=\{ \C^\lambda_H \ | \ \lambda\in (c_H,\infty)\}$  foliates
the   non-simply-connected component of \, $\mathbb{H}^3\setminus \C^{c_H}_H$.
Note that by construction, the elements in $\cF_H$ are
embedded and form a continuous family with respect to
$\lambda$. Therefore, if for $\lambda_0\in (c_H,\infty)$
the spherical $H$-catenoids $\C^{c_H}_H$ and
$\C^{\lambda_0}_H$ are disjoint, then an application
of the maximum principle gives that
$\{ \C^\lambda_H \ | \ \lambda\in (c_H,\lambda_0)\}$
is a foliation of the region between  $\C^{c_H}_H$ and
$\C^{\lambda_0}_H$. Therefore, it suffices to show that
for any $\lambda$ sufficiently large, $\C^{c_H}_H$
and $\C^{\lambda}_H$ are disjoint.

Since $\lim_{\lambda\to\infty}d_H(\lambda)=0$, the
end points of the generating curve $\beta _H^\l$ of
$\C^{\lambda}_H$ converge to $(0,0,1)$
in $\partial_\infty \HH^3$ as $\lambda$ goes to infinity.
Recall that the mean curvature vector of $\C_H^\lambda$
is pointing toward the $x$-axis. Using barriers that are planes
of constant mean curvature $H$ and are rotationally
invariant with respect to rotations around the $x$-axis,
it can be shown that as
$d_H(\l)\to 0$, the set $\C^{\lambda}_H$ viewed in the unit ball
converges to the unit
circle in the $(y,z)$-plane. Thus, for any $\lambda$
sufficiently large, $\C^{c_H}_H$ and $\C^{\lambda}_H$
are disjoint and, by the previous discussion, the
family of spherical $H$-catenoids $\cF_H$  foliates
the non-simply-connected component of \,
$\mathbb{H}^3\setminus \C^{c_H}_H$ when $H\in(0,1)$.

The foliation in the $H=0$ case can be obtained
as the limit as $H$ goes to zero of the foliations $\cF_H$.
\end{proof}

\begin{remark} \label{rem:unstable} {\em Since,  for any fixed $H\in [0,1)$ and $\l<c_H$, the
Jacobi function $J_H^{\l}$ on $\C^\l_H$ induced from the variational
vector field of  the Gomes family $\C^{\l}_H$ is positive
along the circle in $\C_H^\l$ closest to its axis of revolution but limits to
$-\infty$ at its asymptotic boundary, then such a $\C^\l_H$ is unstable.  Thus,
the pair of asymptotic boundary circles of $\C^\lambda_H$, $\l<c_H$,
bounds  two spherical $H$-catenoids, a stable one $\C^{\l_2}$,
and an unstable one $\C^{\l_1}$, where $\l_1<\l_2$.
}\end{remark}

We will need the next proposition in the proof of Theorem~\ref{main}.
\begin{proposition} \label{prop:5.3}
Fix $H\in [0,1)$.  For $t\in [0,\infty)$, let $S(t)$ be the parallel surface in
$\HH^3$ that lies ``above"  $P_{y,z}$ at distance $t$ and for $t\in (-\infty,0)$
let $S(t)$ denote the parallel surface that lies ``below"   $P_{y,z}$ at distance $-t$.
Let $V_x$ denote the related Killing field generated by the
one-parameter group of isometries $\{\phi_\ve\colon\HH^3\to \HH^3 \mid \ve\in \R\} $
where $\phi_\ve\colon\HH^3\to \HH^3$ is the hyperbolic translation
along the $y$-axis by the signed
distance $\ve$. Then:
\ben[1.]
\item For each $t\in \R$, $S(t)$ intersects  $\C_H^{c_H}$ transversely
in a single circle (closed curve of constant geodesic curvature in $S(t)$)
centered at the intersection of the $x$-axis with $S(t)$.
\item If we let $\C_H^{c_H}(+)$ denote the portion of
$\C_H^{c_H}$ with non-negative $y$-coordinate,
then  $\C_H^{c_H}(+)$ is a $V_y$-Killing graph over its
projection to $P_{x,z}$, and this projected
domain has boundary $\C_H^{c_H}\cap P_{x,z}$.
\item Let $D(t)\subset S(t)$ be the open disk
bounded by the circle $S(t)\cap \C_H^{c_H}$.
For $\ve<0$ sufficiently close to zero,
$\phi_{\ve}(x\mbox{\rm -axis})$ intersects
 each
of the disks in a single point.
\item For $\ve$ satisfying the previous item and for
$\l_2\in (c_H,\infty)$ sufficiently close to $c_H$,
$\displaystyle \phi_{\ve}(\C_H^{c_H})\cap \cup_{s \in [c_H,\l_2]}\C_H^s$
is a pair of infinite
strips in $ W=\cup_{s\in [c_H,\l_2]}C_H^s$ that
separate $ W$ into two regions,
and for one of these two regions the portion of
$ \phi_\ve(\C_H^{c_H})$ in its boundary has
non-negative mean curvature with respect to the
inward pointing to the boundary.
 \een
\end{proposition}
\begin{proof}
Since the boundary circle of $P_{y,z}$ has linking
number 1 with the arc $\beta_H^{c_H}$ and the $x$-axis,
each of the planes $S(t)$ intersect $\beta_H^{c_H}$ in some
point, and so since these planes are also invariant under
rotation around the $x$-axis, each point in
$\beta_H^{c_H}\cap S(t)$ lies on a circle of intersection of
$S(t)$ with  $\C_H^{c_H}$.
Also note that $S(0)=P_{y,z}$ intersects  $\C_H^{c_H}$ orthogonally in
a circle of radius $c_H$ centered at $(0,0,0)$.
If  $S(t)$ fails to intersect $\C_H^{c_H}$
transversely at some point, then $V_y$ is tangent
to  $\C_H^{c_H}$ along some circle $S$
in  $\C_H^{c_H}$ as well as to the related circle in
$S'\subset \C_H^{c_H}$ obtained by reflection in
the plane $P_{y,z}$. But in this case the compact
subannulus of $ \C_H^{c_H}$ bounded by $S\cup S'$
would represent a compact non-strictly stable
subdomain of  $\C_H^{c_H}$, which contradicts that $ \C_H^{c_H}$
is stable. This contradiction shows that $S(t)$ always intersects
$ \C_H^{c_H}$ transversally. Since $S(0)$ and $ \C_H^{c_H}$ intersect
along a single circle, elementary arguments imply
that every $S(t)$ intersects $ \C_H^{c_H}$
transversely in a single circle centered at the point of
intersection of $S(t)$ with the $x$-axis.
This finishes the proof of item~1.

Items~2 and 3  follow from
item~1 and the details will be left to the reader.
Finally, item~4 follows  from item~3 after  observing that for
$\l_2$ chosen sufficiently close to $c_H$, then
for all $t\in \R$, $\phi_\ve(\C_H^{c_H})\cap S(t)$ is a translation
of the  circle $\C_H^{c_H}\cap S(t)$ and these two
circles intersect transversely in two points.
\end{proof}

\nocite{cosk2,na1}

\bibliographystyle{plain}
\bibliography{bill}
\end{document}